\documentclass[12pt]{article}

\usepackage{amssymb,amsmath,amsfonts,enumerate}

\title{
\author{{\bf Lech Pasicki}
\footnote{E-mail: pasicki@agh.edu.pl, \hspace{0.1cm} https://home.agh.edu.pl/$\sim$pasicki}}
\bf{A criterion for Cauchy sequences in db-metric spaces}}

\newtheorem{theorem}{Theorem}%[section]
\newtheorem{lemma}[theorem]{Lemma}

\newtheorem{definition}[theorem]{Definition}

\newcommand{\Xr}{\mbox{$(X,\varrho)$}}
\newcommand{\oi}{\mbox{$[0,\infty)$}}
\newcommand{\xyX}{\mbox{$x,y \in X$}}
\newcommand{\xnn}{\mbox{$(x_{n})_{n \in \mathbb{N}}$}}
\newcommand{\li}{\mbox{$\lim_{n \to \infty}$}}
\newcommand{\limn}{\mbox{$\lim_{m,n \to \infty}$}}
\newcommand{\m}{\mbox{$_{m}$}}
\newcommand{\mpe}{\mbox{$_{m+p}$}}
\newcommand{\n}{\mbox{$_{n}$}}
\newcommand{\nl}{\mbox{$_{n+1}$}}
\newcommand{\nkp}{\mbox{$_{n+kp}$}}
\newcommand{\nklp}{\mbox{$_{n+(k+1)p}$}}
\newcommand{\np}{\mbox{$_{n+p}$}}
\newcommand{\nq}{\mbox{$_{n+q}$}}

\newcommand{\rxnm}{\mbox{$\varrho(x_{n},x_{m})$}}

\newcommand{\no}{\mbox{$n_{0}$}}
\newcommand{\mo}{\mbox{$m_{0}$}}
\newcommand{\mN}{\mbox{$\mathbb{N}$}}
\newcommand{\nN}{\mbox{$n \in \mathbb{N}$}}

\newcommand{\fn}{\mbox{$f^{n}$}}

\begin{document}
\maketitle
%%\vspace{1 in}

\begin{abstract}
\par In this note a criterion for Cauchy sequences is proved which refines the one presented in `Cauchy sequences 
in b-metric spaces', Topology Appl. 373 (2025) 109477.
\end{abstract}

\vspace{0.2 in} 

\par First, let us recal the following (maybe original) idea from \cite{CachSeq}
%Definition 2a.1
\begin{definition}
\label{Def2a.1}
A mapping $\varrho \colon X \times X \to \oi$ is a \textit{db-metric} (\textit{dislocated b-metric}) for $X$ if 
the following conditions are satisfied
\begin{subequations}
\label{con1}
\begin{align}
 &\varrho(x,y) = 0 \textit{ yields } x = y, \quad \xyX, \label{con1a}\\
 &\varrho(x,y) = \varrho(y,x), \quad \xyX, \label{con1b}\\
 &\varrho(x,z) \leq s[\varrho(x,y) + \varrho(y,z)], \quad x,y,z \in X, \textit{ for an } s \geq 1. \label{con1c}
\end{align}
\end{subequations}
Then $\Xr$ is a \textit{db-metric space} (\textit{dislocated b-metric space}).
\end{definition}
\par The mapping presented above is a natural extension of the dislocated-metric, a notion due to Hitzler and Seda 
\cite{HS} ($s=1$). For example, each partial metric (introduced by Matthews \cite[Definition 3.1]{Matt}), if 
nonnegative, is a d-metric. As regards b-metric, assumed is the equivalence in \eqref{con1a}. It should be noted 
that the idea of a b-metric was presented by Czerwik in \cite{Cze} (for $s=2$).
\par Now, let us prove our criterion, which refines \cite[Lemma 2.3]{CachSeq}.
%Lemma 2a.2
\begin{lemma}
\label{Le2a.2}
For $X = \{x\n \colon \nN\}$ and $\varrho \colon X \times X \to \oi$ let \eqref{con1b}, \eqref{con1c} be 
satisfied. Then $\limn \varrho(x\n,x\m) = 0$ if and only if the following conditions hold 
\begin{subequations}
\label{con2}
\begin{align}
&\li \varrho(x\nl,x\n) = 0, \label{con2a}\\
\begin{split}
& \textit{for each small } \delta > 0 \textit{ there exist } \no,p \in \mN \textit{, and  } 0<\lambda<1 
         \textit{ such that } \\ 
& 0 < \varrho(x\n,x\m) < \delta \textit{ yields } \varrho(x\np,x\mpe) < \delta\lambda/s, \quad 
  m,n > \no. \label{con2b}
\end{split} 
\end{align}
\end{subequations}
\end{lemma}
{\bf Proof}
If $\limn \varrho(x\n,x\m) = 0$, then \eqref{con2a} obviously holds, and for a $p \in \mN$ and all $m,n \in \mN$ 
we have $\varrho(x\np,x\mpe) < \delta\lambda/s$, which means that condition \eqref{con2b} is satisfied.
\par Let $\delta > 0$ be arbitrary. Now, for a $p$ as in \eqref{con2b} and $q \in \{2,\ldots,p\}$ (if $p \geq 2$) 
conditions \eqref{con1c}, \eqref{con2a} imply
\begin{equation*}
\begin{split}
 &\varrho(x\nq,x\n) = \varrho(x\n,x\nq) \leq s\varrho(x\n,x\nl) + s\varrho(x\nl,x\nq) \leq \cdots \leq \\
 &s\varrho(x\n,x\nl) + \cdots + s^{q-1}\varrho(x_{n+q-1},x\nq) \to 0, \textit{ for } n \to \infty.
\end{split}
\end{equation*}
The same holds for $q=0$, as $\varrho(x\n,x\n) \leq 2s\varrho(x\nl,x\n)$. Consequently (see \eqref{con2a}), we get 
\begin{equation}
\label{con3}
\varrho(x\nq,x\n) < \delta(1-\lambda)/s, \textit{ for all } n > \mo \geq \no, \textit{ and } q \in \{0,\ldots,p\}.
\end{equation}
We may put $\no = \mo$ in \eqref{con2b}. 
\par Now, assume
\[
 \varrho(x\nkp,x\n) < \delta, \textit{ for a } k \geq 1. 
\]
If $\varrho(x\nkp,x\n) = 0$, then \eqref{con1c}, \eqref{con3} yield
\begin{equation*}
\begin{split}
 &\varrho(x\nklp,x\n) \leq s\varrho(x\nklp,x\nkp) + s\varrho(x\nkp,x\n) = \\
 &s\varrho(x\nklp,x\nkp) < \delta(1 - \lambda) < \delta.
\end{split}
\end{equation*}
For $0 < \varrho(x\nkp,x\n) < \delta$, from \eqref{con2b} and \eqref{con3} we obtain
\[
 \varrho(x\nklp,x\n) \leq s\varrho(x\nklp,x\np) + s\varrho(x\np,x\n) < \delta\lambda + \delta(1-\lambda) = \delta.
\]
Thus, by induction we get
\begin{equation}
\label{con4}
 \varrho(x\nkp,x\n) < \delta, \quad k \in \mN.
\end{equation}
\par Any $m \in \mN$, $m \geq n$ can be written as $m = n + kp + q$, for some $k \in \{0\} \cup \mN$, 
$q \in \{0,1,\ldots,p-1\}$. Now, by \eqref{con3}, \eqref{con4}, for $m,n > \no$ we obtain
\[
 \varrho(x\m,x\n) = \varrho(x_{n+kp+q},x\n) \leq s\varrho(x_{n+kp+q},x\nkp) + s\varrho(x\nkp,x\n) < 
    \delta(1 - \lambda) + s\delta,
\]
which means that $\limn \rxnm = 0$.
$\quad \square$
\par It should be noted that condition \eqref{con2a} cannot be disregarded. Let us consider $x\n =n$, $\nN$. 
Then $0 < \varrho(x\n,x\m) = |n -m| < \delta$ never holds for $\delta < 1$. Consequently, condition \eqref{con2b} 
is satisfied, while $\xnn$ is not a Cauchy sequence.
\par Assume that a selfmapping $f \colon X \to X$ on a db-metric space $\Xr$ is such that for a $c < 1$
\[
  \varrho(fx,fy) \leq c\varrho(x,y), \quad \xyX
\]
holds. Then for $x\n = \fn x$, $\nN$ we have
\[
  \varrho(x\np,x\mpe) \leq c^{p}\varrho(x\n,x\m),
\]
and \eqref{con2b} is satisfied for any $p$ such that $c^{p} < \lambda/s$. Thus, the Banach fixed point 
theorem is an almost immediate consequence of Lemma \ref{Le2a.2}.
\par In \cite{CachSeq} some fixed point theorems are presented. Maybe, Lemma \ref{Le2a.2} will be a useful tool 
to obtain more advanced results.
% \section*{Acknowledgements}
% This work was partially supported by the Faculty of Applied Mathematics AGH UST statutory tasks within subsidy 
% of the Polish Ministry of Science and Higher Education, grant no. 16.16.420.054.

\end{document}